\documentclass[10pt,reqno]{amsart}
\usepackage{amssymb}
\usepackage{amsmath}
\theoremstyle{definition}
\newtheorem{theorem}{Theorem}[section]
\newtheorem{proposition}[theorem]{Proposition}
\newtheorem{lemma}[theorem]{Lemma}
\newtheorem{corollary}[theorem]{Corollary}
\theoremstyle{definition}

\theoremstyle{remark}
\newtheorem{remark}[theorem]{Remark}

\newcommand{\norm}[1]{\left\lVert#1\right\rVert}

\usepackage{mathtools}

\allowdisplaybreaks[1]

\numberwithin{equation}{section}

\begin{document}

\title[Spectrum of Band Matrices Generated by Oscillatory Sequences]{Spectrum and Fine Spectrum of Band Matrices Generated by Oscillatory Sequences}

%    Remove any unused author tags.
\author{Jyoti Rani}
\address{Department of mathematics, Indian Institute of Technology Bhilai, Durg - 491001, India.}
\curraddr{}
\email{jyotir@iitbhilai.ac.in}

%    author one information
\author{Arnab Patra}
\address{Department of mathematics, Indian Institute of Technology Bhilai, Durg - 491001, India.}
\curraddr{}
\email{arnabp@iitbhilai.ac.in}
%\thanks{}

\author{P D Srivastava}
\address{Department of mathematics, Indian Institute of Technology Bhilai, Durg - 491001, India.}
\curraddr{}
\email{pds@iitbhilai.ac.in}



\begin{abstract} 
In this paper, a new class of band matrices is considered where the entries of each non-zero band form a sequence with two limit points. The compact perturbation technique is used to study the spectrum over the $\ell_{p}, (1<p<\infty)$ sequence space. Several spectral subdivisions such as fine spectrum, discrete spectrum, essential spectrum, etc. are obtained. In addition, a few sufficient conditions on the absence of point spectrum over the essential spectrum are also discussed.
\end{abstract}
\maketitle

\section{\textbf{Introduction}} 
The spectral analysis of infinite matrices, in particular band matrices, defined over sequence spaces has been treated by many researchers worldwide. Over the last few decades, localization of spectrum and various spectral sub-divisions of band matrices over sequences spaces generated by difference equations has evolved into a substantial area of study in the field of spectral theory.

The spectral properties of the double band operator $B(r,s)$ defined by the difference equation with constant coefficient
\[(B(r,s)x)_n = sx_{n-1} + rx_n, \ n \in \mathbb{N}\]
with $x_0=0$ were studied by Altay and Ba{\c{s}}ar \cite{c0c}, Furkan et al. \cite{l1bv}, Bilgi{\c{c}} and Furkan \cite{lpbvp} over the sequence spaces $c_0$ and $c$, $l_1$ and $bv$, $l_p$ and $bv_p$ $(1<p< \infty)$ respectively. The case $r=1=s$ was explored by Altay and Ba{\c{s}}ar \cite{deltac0c}, Kayaduman and Furkan \cite{deltal1bv}, Akhmedov and Ba{\c{s}}ar \cite{deltabvp} over the sequence spaces $c_0,$ $c$ and $l_1,$ $bv$ and $bv_p (1 \leq p < \infty)$ respectively. $B(r,s)$ was further generalised to the triple band operator $B(r,s,t)$ and  Furkan et al. \cite{c0c3,lpbvp3}, Bilgi{\c{c}} and Furkan \cite{l1bv3} examined the spectrum and fine spectrum of $B(r,s,t)$. The fine spectrum of the generalised $n$-band triangular Toeplitz operator was studied by Altun \cite{altun}, Birbonshi and Srivastava \cite{bir}. Band matrices with non-constant band are also been considered in the literature. The spectrum and fine spectrum of the generalised differnece operators $\vartriangle_v$ and $\vartriangle_{uv}$ were obtained in \cite{akhmedov2012some,v_l1} and \cite{ab_c,ab_lp,abc0,el2018spectra,uv_l1} respectively where 
\[(\vartriangle_v x)_n = v_n x_n - v_{n-1}x_{n-1},\]
\[(\vartriangle_{uv} x)_n = u_n x_n - v_{n-1} x_{n-1}\]
with $x_{-1} = 0,$ $n \in \mathbb{N}$ and certain assumptions on the sequences $u$ and $v$. Later on, El-Shabrawy \cite{ab_lp,abc0} has studied the fine spectrum of the operator $\vartriangle_{ab}$ on the sequence space $l_p$ $(1< p< \infty)$ and $c_0$ respectively. In 2012, the fine spectrum of lower triangular triple band matrix $\Delta^2_{uvw}$ on $l_1$ is studied by Panigrahi and Srivastava \cite{pds1} and analogously, the upper triangular case is studied by Altundag and Abay \cite{filomat1}.
An important class of tridiagonal matrices, known as Jacobi matrices, has been studied by several researchers. In this context, we refer the readers to the recent works done by J. Dombrowski \cite{dombrowski2011commutator, dombrowski2017some,dombrowski2002absolute}.
 The generalised $m+1$ banded matrix $\Delta^m,$ $m \in \mathbb{N}$ is considered in \cite{durna2023spectral}. The spectrum and fine spectrum of the difference operator $\vartriangle^r_v$ over the sequence spaces $c_0$ and $l_1$ have been studied by Dutta and Baliarsingh (see \cite{rv_c0, rv_l1}). Recently Meng and Mei \cite{meng2019matrix,meng2020spectrum} characterise the spectrum of the generalised difference operator $B_v^{(m)}.$ Spectra of tridiagonal matrices \cite{bilgic2019fine,el2020spectra,karakaya2012fine} and symmetric $2n+1$ band matrices \cite{altun2012fine} are also studied.
For a detailed review, one may refer to the survey articles \cite{baliarsingh2021survey,yecsilkayagil2019survey} and the references therein.

An interesting problem is to study the spectrum and fine spectrum of banded matrices acting over sequence spaces where the entries of the band forms oscillatory sequences. Recent articles \cite{das2017spectrum,mahto2021spectra,patra2020spectrum} focus in this direction. However, the spectral properties of symmetric band matrices whose bands are generated by oscillatory sequences are not much explored. In this article, we attempt to study the spectral properties of a class of penta-diagonal band matrices defined over the sequence space $\ell_p (1 < p < \infty)$ where the entries in the non-zero bands form sequences with two limit points. The case where the entries in the non-zero bands form oscillatory sequences is treated separately.

Let $\ell_{p}$ represents the Banach space of  $p$-absolutely summable sequences of real or complex numbers with the norm 
\[ \norm{x}_{p}=\left( \sum_{n=1}^{\infty}|x_n|^p\right)^{\frac{1}{p}}.\] 
Also let $\mathcal{D}_p$ denotes the set of all diagonal operators on $\ell_p.$ For any operator $T \in \mathcal{D}_p,$ $\mbox{diag}(T)$ represents the sequence in the diagonal of $T.$ In this work we investigate the spectral properties of a class of operators $T$ defined over $\ell_p$ represented by the following form:
\[T = S_r^2D_1 + D_2S_{\ell}^2 + D_3\]
where $S_r, S_{\ell} : \ell_p \to \ell_p$, denotes the right shift operator, left shift operator respectively and $D_1, D_2, D_3 \in \mathcal{D}_p$ with $\mbox{diag}(D_1) = \{c_n\} \subset \mathbb{C} \setminus \{0\}$ $\mbox{diag}(D_2) = \{b_n\} \subset \mathbb{C} \setminus \{0\}$ and $\mbox{diag}(D_3)=\{a_n\} \subset \mathbb{C}$. We further assume that the subsequences $\{a_{2n-1}\},$ $\{b_{2n-1}\}$, $\{c_{2n-1}\}$ converges to the non-negative real numbers $r_1,$ $s_1,$ $s_1$ respectively and $\{a_{2n}\},$ $\{b_{2n}\}$, $\{c_{2n}\}$ converges to the non-negative real numbers $r_2,$ $s_2,$ $s_2$ respectively where $s_1 \neq 0$ and $s_2 \neq 0$.  

Our focus is to investigate the spectral properties of the operator $T$ using compact perturbation technique. Let us consider another operator $T_0$ over $\ell_p$ defined by
\[T_0 = S_r^2D_1^\prime + D_2^\prime S_{\ell}^2 + D_3^\prime\] 
where $D_1^\prime, D_2^\prime, D_3^\prime \in \mathcal{D}_p$ with 
\[\mbox{diag}(D_1^\prime) =\mbox{diag}(D_2^\prime) =  \{s_1, s_2, s_1, s_2, \cdots\}, \ \mbox{diag}(D_3^\prime) = \{r_1, r_2, r_1, r_2, \cdots\}.\]
Using the properties of compact operators, it can be proved that $T - T_0$ is a compact operator over $\ell_{p}$. Both the operators $T$ and $T_0$ can be represented by the following penta-diagonal matrices
\[T=\begin{pmatrix}
	a_1 &0& b_1 &0&0&\cdots &  \\
	0 & a_2&0 & b_2& 0&\cdots&\\
	c_1&0&a_3&0&b_3&\cdots&\\
	0&c_2&0&a_4&0&\cdots\\
	0&0&c_3&0&a_5&\cdots\\
	\vdots  &\vdots & \vdots  & \vdots &\vdots&\ddots
	
	\end{pmatrix},
	T_0=
		\begin{pmatrix}
			r_1 & 0 &s_1&0&0&\cdots &  \\
			0 & r_2 & 0& s_2&0& \cdots&\\
			s_1&0&r_1&0&s_1&\cdots&\\
			0&s_2&0&r_2&0&\cdots&\\
			0&0&s_1&0&r_1&\cdots\\
			\vdots  & \vdots  & \vdots &\vdots&\vdots& \ddots
		\end{pmatrix}.
	\]
 
  We obtain the spectrum, fine spectrum and the sets of various spectral subdivisions of the operator $T_0.$ It is interesting to note that the spectrum of $T_0$ is given by \[[r_1-2s_1,r_1+2s_1]\cup[r_2-2s_2,r_2+2s_2],\] which is also its essential spectrum with no eigenvalues. Later we investigate how the spectrums of $T$ and $T_0$ are related. Few results on the essential spectrum of $T$ (which is identical to the essential spectrum of $T_0$) being devoid of its eigenvalues are also derived. This helps us to characterize the point spectrum of $T.$ Theory of difference equations plays an important role in our study. We use various results on the asymptotic behaviour of solutions of difference equations to demonstrate the findings of our paper. For more details on difference equations one can refer \cite{difference}.

The remainder of paper is organized as follows: section 2 is devoted to introduce some terminologies and results which are relevant to our work. Section 3 contains the results on the spectrum and fine spectrum of $T_0$ over $\ell_p.$ The spectral properties of $T$ are discussed in section 4.

 \section{\bf Preliminaries}
  Let $X$ and $Y$ are Banach spaces and for any operator $A:X \to Y$, $N(A)$ and $R(A)$ denote the null space and range space of $A$ respectively. The operator $A^*:Y^* \to X^*$ is called the adjoint operator and defined by
 \begin{equation*}
 	(A^{*}f)(x)=f(Ax) \quad \text{for all}~ f \in Y^{*} \text{and}~ x\in X
 \end{equation*}
 where $X^*$, $Y^*$ are the dual spaces of $X$ and $Y$ respectively. $B(X)$ denotes the set of all bounded linear operators from $X$ to itself. For any $A\in B(X)$, the resolvent set $ \rho(A, X)$ of $A$ is the set of all $\lambda $ in the complex plane such that $(A-\lambda I)$ has a bounded inverse in $X$ where $I$ is the identity operator defined over $X$. The complement of resolvent set in the complex plane $\mathbb{C} $ is called the spectrum of $A$ and it is denoted by $ \sigma(A, X)$. The spectrum $\sigma(A,X)$ can be partitioned into three disjoint sets which are 
 \begin{enumerate}
 	\item[(i)] the point spectrum, denoted by $\sigma_p(A, X)$, is the set of all such $\lambda \in \mathbb{C}$ for which $(A-\lambda I)^{-1}$ does not exist. An element $\lambda\in\sigma_p(A , X)$ is called an eigenvalue of $A$, 
 	\item[(ii)] the continuous spectrum, denoted by $\sigma_c(A , X)$, is the set of all such $\lambda \in\mathbb{C}$ for which $(A-\lambda I)^{-1}$ is exists, unbounded and $R(A-\lambda I)$ is dense in $X$ but $R(A-\lambda I)\neq X$,
 	\item[(iii)] the residual spectrum, denoted by $\sigma_r(A, X)$, is the set of all such $\lambda \in\mathbb{C}$ for which $(A-\lambda I)^{-1}$ exists (and may be bounded or not) but $R(A-\lambda I)$ is not dense in $X$.
 \end{enumerate} 
These three disjoint sets are together known as fine spectrum and their union becomes the whole spectrum. There are some other important subdivisions of the spectrum such as approximate point spectrum $\sigma_{app}(A, X)$, defect spectrum $\sigma_{\delta}(A, X)$ and compression spectrum $\sigma_{co}(A, X)$, defined by

\begin{eqnarray*}
	\sigma_{app}(A, X)&=&\{\lambda \in \mathbb{C}:~~ \text{there exists a Weyl sequence for}~~(A - \lambda I)\},\\
	\sigma_{\delta}(A,X)&=&\left\lbrace \lambda \in \mathbb{C}: (A-\lambda I ) \text{ is not surjective} \right\rbrace,\\
	\sigma_{co}(A,X)&=&\left\lbrace  \lambda \in \mathbb{C}:\overline{R(A-\lambda I)} \neq X\right\rbrace.
\end{eqnarray*}
 The sets which are defined above also forms subdivisions of spectrum of $A$ (which are not necessarily disjoint) as follows \cite[p.178]{basar2012summability}
 \begin{eqnarray*}
 	\sigma(A,X)&=&\sigma_{app}(A,X) \cup \sigma_{co}(A,X),\\
 	\sigma(A,X)&=&\sigma_{app}(A,X) \cup \sigma_{\delta}(A,X).
 \end{eqnarray*}
  %Approximate point spectrum and compression spectrum also form a subdivision $\sigma(T,X)=\sigma_{ap}(T,X) \cup \sigma_{co}(T,X)$ of spectrum of $T$.
   An operator $A \in B(X)$ is said to be Fredholm operator if $R(A)$ is closed and  $\dim(N(A))$, $\dim(X / R(A))$ are finite. In this case the number
  \[ \dim(N(A))- \dim(X / R(A)) \]
  is called the index of the Fredholm operator $A$. The essential spectrum of $A$ is defined by the set \[ \sigma_{ess}(A,\ell_{p})=\left\lbrace \lambda \in \mathbb{C}:(A-\lambda I)~ \text{is not a Fredholm operator}\right\rbrace.\]
   If $A$ is a Fredholm operator and $K \in B(X)$ is a compact operator then $A+K$ is also a Fredholm operator with same indices. Since compact perturbation does not effect the Fredholmness and index of a Fredholm operator, we have \[\sigma_{ess}(A, X)=\sigma_{ess}(A+K, X).\]

For any isolated eigenvalue $\lambda$ of $A,$ the operator $P_A$ which is defined by
\[P_A(\lambda) = \frac{1}{2 \pi i} \int_\gamma  (\mu I - A)^{-1} d \mu,\]   
is called the Riesz projection of $A$ with respect to $\lambda$ where $\gamma$ is positively orientated circle centred at $\lambda$ with sufficiently small radius such that it excludes other spectral values of $A.$ An eigenvalue $\lambda$ of $A$ is said to be a discrete eigenvalue if it is isolated and the rank of the associated Riesz projection is finite. The rank of the Riesz projection is called the algebraic multiplicity of $\lambda$. The set of all such eigenvalues with finite multiplicities is called the discrete spectrum of $A$ and it is denoted by $\sigma_d(A,X).$ This type of eigenvalues sometimes referred as eigenvalues with finite type.

In the following proposition, we mention some inclusion relation of spectrum of a bounded linear operator and its adjoint operator.
\begin{proposition} \label{p3.1} \cite[p.195]{basar2012summability}
	If $X$ is a Banach space and $A\in B(X)$, $A^*\in B(X^*)$ then the spectrum and subspectrum of $A$ and $A^*$ are related by the following relations: 
	\begin{itemize}
	\item[(a)] $\sigma(A^*,X^*)=\sigma(A,X).$
	\item[(b)] $\sigma_c(A^*,X^*)\subseteq\sigma_{app}(A,X).$
	\item[(c)] $\sigma_{app}(A^*,X^*)=\sigma_{\delta}(A,X).$
	\item[(d)] $\sigma_{\delta}(A^*,X^*)=\sigma_{app}(A,X).$
	\item[(e)] $\sigma_p(A^*,X^*)=\sigma_{co}(A,X).$
	\item[(f)] $\sigma_{co}(A^*,X^*)\supseteq\sigma_{p}(A,X).$
	\item[(g)] $ \sigma(A,X)=\sigma_{app}(A,X)\cup\sigma_{p}(A^*,X^*)$=$\sigma_p(A,X)\cup\sigma_{app}(A^*,X^*).$
	\end{itemize}
\end{proposition} 
%\begin{proposition}\label{p3.2}\cite[p.373]{gohberg2013classes}
%Let $T : X\rightarrow X$ be an operator with a non-empty resolvent set, and let $\Omega$	be an open connected subset of $\mathbb{C}\setminus \sigma_{ess}(T)$. If $\Omega \cap \rho(T) \neq \emptyset$ then $\sigma(T)\cap \Omega$
%is a finite or countable set, with no accumulation point in $\Omega$, consisting of eigenvalues of $T$ of finite type.
%\end{proposition} 
Here we record few lemmas related to the boundness of an infinite matrix defined over sequence spaces, which are useful to our research.
\begin{lemma}\cite[p. 253]{choudhary1989functional} \label{l1}
The matrix $A=(a_{nk})$ gives rise to a bounded linear operator $T \in B(\ell_{1})$ from $\ell_1$ to itself if and only if the supremum of $\ell_1$ norms of  the columns of $A$ is bounded.
\end{lemma}
\begin{lemma}\cite[p. 245]{choudhary1989functional} \label{l2}
The matrix $A=(a_{nk})$ gives rise to a bounded linear operator $T \in B(\ell_\infty)$  from $\ell_\infty$ to itself if and only if the supremum of $\ell_1$  norms of the rows of $A$ is bounded.
\end{lemma}
\begin{lemma}\cite[p. 254]{choudhary1989functional}\label{L3.5}
The matrix $A=(a_{nk})$ gives rise to a bounded linear operator $T\in B(\ell_p)(1<p<\infty)$ if $T \in B(\ell_1)\cap B(\ell_{\infty}).$
\end{lemma}

\section{\bf Spectra of $T_0$}

  It is already mentioned that we study the spectral properties of $T$ by using the spectral properties of $T_0$ and compact perturbation technique. In this section we derive the spectrum and fine spectrum of $T_0$. The notation $\|T\|_p$ denotes the operator norm of an operator $T \in B(\ell_{p})$ where $1 \leq p \leq \infty$.  
\begin{theorem} The operator
	$T_0:\ell_p \rightarrow \ell_p$ is a bounded linear operator which satisfies the following inequality
  $$	\left(\frac{|r_1|^p+|r_2|^p+|s_1|^p+|s_2|^p}{2}\right)^{\frac{1}{p}}\leq \norm{T_0}_{p}\leq \left( 3^{p-1}\left(|r_1|^p+2|s_1|^p+|r_2|^p+2|s_2|^p\right)\right)^{\frac{1}{p}}.$$
\end{theorem}
\begin{proof}
As linearity of $T_0$ is trivial, we omit it. Let $e=(1,1,0,0,...) \in\ell_{p}$. Then $T_{0}(e)=(r_1,r_2,s_1,s_2,0,...)$ and one can observe that
\begin{equation*}
\frac{\norm{T_0(e)}_{p}}{\norm{e}_{p}}=\left( \frac{|r_1|^p+|r_2|^p+|s_1|^p+|s_2|^p}{2}\right)^{\frac{1}{p}}.
\end{equation*}
This proves \[ \left( \frac{|r_1|^p+|r_2|^p+|s_1|^p+|s_2|^p}{2}\right)^{\frac{1}{p}}\leq \norm{T_0}_{p}.\]

\noindent Also, let $x=\{x_n\}\in \ell_p$ and $x_n=0$ if $n \leq0$.
Then,
\begin{align*} 
\norm{T_0(x)}_{p}^p=&\sum_{n=1}^{\infty}|s_{1}x_{2n-3}+r_1x_{2n-1}+s_1x_{2n+1}|^p+	\sum_{n=1}^{\infty}|s_2x_{2n-2}+r_2x_{2n}+s_2x_{2n+2}|^p\\
\leq& \sum_{n=1}^{\infty}(|s_{1}x_{2n-3}|+|r_1 x_{2n-1}|+|s_1 x_{2n+1}|)^p \\ 
& +\sum_{n=1}^{\infty}(|s_2x_{2n-2}|+|r_2x_{2n}|+|s_2x_{2n+2}|)^p.\\
\end{align*}
 By Jensen’s inequality we get,
\begin{align*}
\norm{T_0(x)}_{p}^p \leq & 3^{p-1}\sum_{n=1}^{\infty}\left( |s_1x_{2n-3}|^p+|r_1x_{2n-1}|^p+|s_1x_{2n+1}|^p\right)\\ +& 3^{p-1}\sum_{n=1}^{\infty}\left( |s_2x_{2n-2}|^p+|r_2x_{2n}|^p+|s_2x_{2n+2}|^p\right)\\
\leq& 3^{p-1}\left(|r_1|^p+2|s_1|^p+|r_2|^p+2|s_2|^p\right){\norm{x}_p^p}.  
\end{align*}
This implies,
$$\norm{T_0} \leq \left( 3^{p-1}\left(|r_1|^p+2|s_1|^p+|r_2|^p+2|s_2|^p\right)\right)^{\frac{1}{p}}.$$ This completes the proof. 
\end{proof}
The following theorem proves the non-existence of eigenvalues of the operator $T_0$ in $\ell_p$.
\begin{theorem}\label{t1.2}
The point spectrum of $T_0$ over $\ell_p$ is given by $\sigma_p(T_0,\ell_p)=\emptyset.$
\end{theorem}
\begin{proof}
Consider $(T_0 - \lambda I)x=0$ for $\lambda\in\mathbb{C}$ and $x=\{x_n\}\in \mathbb{C^N}$. This gives the following system of equations

	\[
		\left. \begin{array}{lr}
		(r_{1} - \lambda) x_{1}+s_{1} x_{3}&=0\\
		(r_{2} - \lambda) x_{2}+s_{2} x_{4}&=0 \\
		s_1 x_1 + (r_1 - \lambda) x_3 + s_1 x_5 &= 0\\
		s_2x_2+(r_2 - \lambda) x_4+s_2x_6&=0 \\
		                       &\vdots \\
    	s_1 x_{2n-1}+(r_1 - \lambda) x_{2n+1}+s_1 x_{2n+3}&=0\\
		s_2 x_{2n}+(r_2-\lambda)x_{2n+2}+s_2 x_{2n+4}&=0\\
		                             &\vdots
	\end{array}\right\}
\]
If $x_1 = 0$ then $x_{2n-1} = 0$ for all $n \in \mathbb{N}$. Similarly $x_2 = 0$ implies $x_{2n} = 0$ for all $n \in \mathbb{N}.$ Therefore let $(x_1, x_2) \neq (0,0)$ and consider two sequences $\{y_n\}$ and $\{z_n\}$ where  $y_n=x_{2n-1}$ and $z_n=x_{2n}$, $n \in \mathbb{N}$ respectively. Then the system of equations of $(T_0-\lambda I)x=0$ reduces to
 \begin{equation}\label{3.1}
  y_{n}+p_1y_{n+1}+y_{n+2}=0,
 % p_1y_{1}+y_{2}=0
 \end{equation}
 \begin{equation} \label{3.2}
 z_{n}+p_2z_{n+1}+z_{n+2}=0,\\
 %p_2z_{1}+z_{2}=0 
 \end{equation}
where $p_1=\frac{r_1-\lambda}{s_1}$, $p_2=\frac{r_2-\lambda}{s_2}$, $n \in \mathbb{N}\cup \{0\}$ and $y_0=z_0=0$. If $x_1 \neq 0,$ the general solution of the difference equation (\ref{3.1}) is given by (\cite[p. 75]{difference},
\begin{equation} \label{eqq1}
y_n = \begin{cases}
& (c_1 + nc_2) (-1)^n, \ \ \mbox{if} \ p_1 = 2\\
& c_1+ n c_2, \ \ \mbox{if} \ p_1 = -2\\
& c_1 \alpha_1^n + c_2 \alpha_2^n, \ \ \mbox{if} \ p_1 \notin \{-2,2\}
\end{cases}
\end{equation}
where $c_1,$ $c_2$ are arbitrary constants and $\alpha_1,$ $\alpha_2$ are the roots of the polynomial 
\begin{equation} \label{ch1}
y^2+p_1y+1=0
\end{equation}
which is called the characteristic polynomial of (\ref{3.1}). The following two equalities 
\begin{equation*}
\alpha_1\alpha_2=1  \text{ and }  \alpha_1+\alpha_2= -p_1
\end{equation*}
are useful.
Equation (\ref{eqq1}) suggests there are three cases to be considered.\\

\noindent Case 1: If $p_1=2$ (i.e., $\lambda =r_1-2s_1 $). In this case the general solution of (\ref{3.1}) is
 \begin{equation*}
 	y_{n}=(c_1+c_2n)(-1)^{n}, \ n \in \mathbb{N}\cup\{0\}
 \end{equation*}
with the initial condition $y_{0}=0$ which gives $c_1 = 0.$ This reduces the solution as $y_n = nc_2 (-1)^n.$ This also implies $c_2 = -y_1$ and the solution in this case is 
\[y_n = ny_1 (-1)^{n+1}, ~~ n \in \mathbb{N}.\]

\noindent Case 2: If $p_1=-2$ (i.e., $\lambda =r_1+2s_1 $). Similar as Case 1, the solution reduces to
\[y_n = ny_1,~~ n \in \mathbb{N}.\]

\noindent Case 3: If $p_1 \notin \{-2, 2\}.$ The general solution of (\ref{3.1}) is given by
 \begin{equation*}
 	y_{n}=c_1 \alpha_1^n+c_2 \alpha_2^n.
 \end{equation*}
With the help of initial condition $y_0=0$ and by using the equalities $\alpha_1\alpha_2=1$, $\alpha_1+\alpha_2= -p_1$, one can obtain that $c_2 = -c_1$ and the solution reduces to 
\[y_n = \frac{\alpha_1^n - \alpha_2^n}{\alpha_1 - \alpha_2} y_1,~~ n \in \mathbb{N}.\]
If $y_1 \neq 0$ then  $\{y_n\} \notin \ell_p$ in Case 1 and Case 2. In Case 3 $\{y_n\} \in \ell_p$ if and only if $|\alpha_1| <1$ and $|\alpha_2|<1$ which can not be the case since $\alpha_1 \alpha_2  = 1.$ Hence in all the three cases $\{y_n\} \in \ell_p$ if and only if $y_1 = 0$ and this leads to the fact that $x_1 \neq 0$. Hence there is no non-trivial solution of ({\ref{3.1}}).

Similarly for the difference equation (\ref{3.2}), if $x_2 \neq 0,$ the general solution $\{z_n\}$ is of the form
\begin{equation} \label{eqq2}
z_n = \begin{cases}
& (d_1 + nd_2) (-1)^n, \ \ \mbox{if} \ p_2 = 2\\
& d_1+ n d_2, \ \ \mbox{if} \ p_2 = -2\\
& d_1 \beta_1^n + d_2 \beta_2^n, \ \ \mbox{if} \ p_2 \notin \{-2,2\}
\end{cases}
\end{equation}
where $d_1,$ $d_2$ are arbitrary constants and $\beta_1,$ $\beta_2$ are the roots of the polynomial 
\begin{equation} \label{ch2}
z^2+p_2z+1=0
\end{equation}
which is called the characteristic polynomial of difference equation (\ref{3.2}). In a similar way, it can be proved that $\{z_n\} \in \ell_p$ if and only if $z_1 = 0$ and this leads to the trivial solution of (\ref{3.2}). Hence, there does not exist any non-trivial solution of the system $(T_0 - \lambda I)x = 0$ such that $x \in \ell_p.$ This proves the required result.
\end{proof}

\begin{remark}
The solution $x = \{x_n\}$ of the system $Tx = \lambda x,$ which are obtained in terms of the sequences $\{y_n\}$ and $\{z_n\}$ in the equations (\ref{eqq1}) and (\ref{eqq2}) respectively, actually depends on the unknown $\lambda.$ Therefore, instead of writing $x_n(\lambda),$ we write $x_n$ for the sake of brevity throughout this paper except in Theorem \ref{t4.12} where the dependency of the solutions on $\lambda$ is vital. 
\end{remark}

The adjoint operator of $T_0$ is $T_0^*$ which is defined over sequence space $\ell_p^*$ where $\ell_p^*$ denotes the dual space of $\ell_p$ which is isomorphic to $\ell_q$ where $\frac{1}{p}+ \frac{1}{q}=1$.
\begin{corollary}\label{c1.3}
	The point spectrum of adjoint operator $T_0^*$ over the sequence space $\ell_p^*$ is given by
	 $\sigma_p(T_0^*,\ell_p^*)=\emptyset.$
\end{corollary}
\begin{proof}
It is well known that the adjoint operator $ {T_0}^*:\ell_p^* \to \ell_p^* $, is represented by transpose of the matrix $T_0$. Since $T_0$ is represented by a symmetric matrix, using the same argument as Theorem \ref{t1.2}, it is easy to prove that $\sigma_p(T_0^*,\ell_p^*)=\emptyset.$
\end{proof}
\begin{corollary} \label{res}
	The residual spectrum of $T_0$ over the sequence space $\ell_{p}$ is given by
$\sigma_r(T_0,\ell_p)=\emptyset$.
\end{corollary}
\begin{proof}
	We know that the operator $T$ has a dense range if and only if $T^*$ is one to one \cite[p.197]{basar2012summability}. Using this we have the following relation
\[\sigma_r(T_0,\ell_p)=\sigma_p(T_0^*,\ell_p^*)\setminus\sigma_p(T_0,\ell_p).\]
Hence, $\sigma_r(T_0,\ell_p)=\emptyset.$
\end{proof}

Following that, we obtain the spectrum of $T_0$.
\begin{theorem} \label{spec}
	The spectrum of $T_0$ over $\ell_{p}$ is given by
	\begin{equation*}
	\sigma(T_0,\ell_p) =\left[r_{1}-2s_1, r_1+2s_1\right] \cup [r_2-2s_2, r_2+2s_2].
	\end{equation*}
\end{theorem}
\begin{proof}
	First we prove the inclusion relation\[\sigma \left(  T_0, \ell_{p}\right) \subseteq\left[ r_{1}-2s_1, r_1+2s_1\right] \cup [r_2-2s_2, r_2+2s_2].\]
	Let $\lambda\notin[r_{1}-2s_1, r_1+2s_1] \cup [r_2-2s_2, r_2+2s_2] $. Since point spectrum of $T_0$ over $\ell_p$ is empty, $(T_0-\lambda I)^{-1}$ exists for all $\lambda \in \mathbb{C}$. Consider the characteristic polynomials (\ref{ch1}) and (\ref{ch2}) as defined in Theorem \ref{t1.2} which has the roots $\alpha_1,$ $\alpha_2$ and $\beta_{1},$ $\beta_{2}$ respectively. Since $\lambda \in \mathbb{C}\setminus \left( \left[r_{1}-2s_1, r_1+2s_1\right] \cup [r_2-2s_2, r_2+2s_2]\right)$, either $|\alpha_1|<1,|\alpha_2|>1$ or $|\alpha_1|>1,|\alpha_2|<1$. Similar argument also applies for the roots $\beta_1$ and $\beta_2$. Without loss of generality, we can assume that $|\alpha_1|<1<|\alpha_2|$ and $|\beta_{1}|<1<|\beta_{2}|$. One can check that  the infinite matrix $B = (b_{nk})$ mentioned below is the inverse of $(T_0-\lambda I)$, where
	\begin{equation*}
			b_{nk} =\begin{cases}
				\frac{1}{{s_1(\alpha_1^2-1)}} \left(\alpha_1^{\frac{n+1}{2}-\frac{k+1}{2}+1} - \alpha_1^{\frac{n+1}{2}+\frac{k+1}{2}+1}\right),
				& \text{if $n, k$ both are odd and $n\geq k$ }\\\\ 
				\frac{1}{s_1(\alpha_1^2-1)} \left( \alpha_1^{\frac{k+1}{2}-\frac{n+1}{2}+1} - \alpha_1^{\frac{k+1}{2}+\frac{n+1}{2}+1} \right), & \text{if $n, k$ both are odd and $n<k$}\\\\
				\frac{1}{s_2(\beta_1^2-1)} \left(\beta_1^{\frac{n}{2}-\frac{k}{2}+1} - \beta_1^{\frac{n}{2}+\frac{k}{2}+1} \right), & \text{if $n, k$ both are even and $n\geq k$}\\\\
				\frac{1}{s_2(\beta_1^2-1)} \left(\beta_1^{\frac{k}{2}-\frac{n}{2}+1} - \beta_1^{\frac{k}{2}+\frac{n}{2}+1} \right), & \text{if $n, k$ both are even and $n<k$}\\\\
				0, & \text{otherwise}.
			\end{cases}
	\end{equation*}
From Lemma \ref{l1} it follows that
\begin{equation*}
	\norm{(T_0-\lambda I)^{-1} }_1=  \sup\limits_{k} \sum\limits_{n=1}^{\infty}|b_{nk}|. %&= \sup\limits_{k}\left( \sum\limits_{n=1}^{k-1}|b_{nk}|+\sum\limits_{n=k}^{\infty}|b_{nk}|\right) 
\end{equation*}
Let for each $k \in \mathbb{N},$ $S_k$ denotes the sum $\sum\limits_{n=1}^{\infty}|b_{nk}|.$ Thus
\[S_k	=\sum\limits_{n=1}^{k-1}|b_{nk}|+\sum\limits_{n=k}^{\infty}|b_{nk}|.\]
Two cases are considered.\\
Case 1: ($k$ is odd) In this case 
\[\sum\limits_{n=1}^{k-1}|b_{nk}| = |b_{1k}| + |b_{3k}| + \cdots + |b_{k-2,k}|\]
as $b_{nk} = 0$ when $n$ is even. Putting $n=2q-1$ where $q$ runs over $1, 2, \cdots , \frac{k-1}{2}$ we obtain
\[\sum_{n=1}^{k-1}|b_{nk}|=\sum_{q=1}^{\frac{k-1}{2}}|b_{2q-1,k}|=\frac{1}{s_1(\alpha_1^2-1)} \sum_{q=1}^{\frac{k-1}{2}}|\alpha_1^{\frac{k+1}{2}-q+1}-\alpha_1^{\frac{k+1}{2}+q+1}|.\]
Now,
\begin{align*}
	\sum_{q=1}^{\frac{k-1}{2}}|\alpha_1^{\frac{k+1}{2}-q+1}-\alpha_1^{\frac{k+1}{2}+q+1}|
\leq&|\alpha_1|^{\frac{k+1}{2}+1}\sum_{q=1}^{\frac{k-1}{2}}|\alpha_1|^{-q}+|\alpha_1|^{\frac{k+1}{2}+1}\sum_{q=1}^{\frac{k-1}{2}}|\alpha_1|^q\\=&|\alpha_1|^{\frac{k+3}{2}}\left( \frac{|\alpha_1|^\frac{1-k}{2}-1}{1-|\alpha_1|}+\frac{|\alpha_1|-|\alpha_1|^\frac{k+1}{2}}{1-|\alpha_1|}\right) \\
=&\frac{|\alpha_1|^2-|\alpha_1|^{\frac{k+3}{2}}+|\alpha_1|^\frac{k+5}{2}-|\alpha_1|^{k+2}}{1-|\alpha_1|}.
\end{align*}\\
As $|\alpha_1|<1$, the above relation gives us 
\begin{equation*}\label{3.4}
 \sum\limits_{n=1}^{k
 	-1}|b_{nk}| <{\infty}. 
\end{equation*} 
Also 
\begin{align*}
	\sum\limits_{n=k}^{\infty}|b_{nk}| =& \sum\limits_{q=0}^{\infty}|b_{k+2q,,k}|\\
									   =& \frac{1}{s_1(\alpha_1^2-1)}\sum_{q=0}^{\infty}\left| \alpha_1^{\frac{k+2q+1}{2}-\frac{k+1}{2}+1}-\alpha^{\frac{k+2q+1}{2}+\frac{k+1}{2}+1}\right| \\
									   =& \frac{1}{s_1(\alpha_1^2-1)}\sum_{q=0}^{\infty}\left| \alpha_1^{q+1}-\alpha_1^{k+q+2}\right|. 
\end{align*}
The inequality $|\alpha_1^{q+1}-\alpha_1^{k+q+2}|\leq|\alpha_1|^{q+1}+|\alpha_1|^{k+q+2}$ and the fact $|\alpha_1|<1$ provides
\[
\sum_{q=0}^{\infty}\left| \alpha_1^{q+1} -\alpha_1^{k+q+2}\right| <\infty, 
\]
which proves $\sum\limits_{n=k}^{\infty}|b_{nk}|<\infty$. Hence we have,
\begin{equation}\label{3.6}
	 \sum_{n=1}^{\infty}|b_{nk}|<\infty, \ \mbox{for odd $k.$}
\end{equation}
Case 2 : ($k$ is even) In this case if $n$ is odd then $b_{nk}=0$. Let $n$ is even and $n=2q$ where $q \in \mathbb{N}.$ Then
\begin{align*}
\sum_{n=1}^{k-1}|b_{nk}|=|b_{2k}|+|b_{4k}|+...+|b_{n-2,k}|
&=\sum_{q=1}^{\frac{k-2}{2}}|b_{2q,k}| \\
& = \frac{1}{s_2(\beta_{1}^2-1)}\sum_{q=1}^{\frac{k-2}{2}}\left| \beta_1^{\frac{k}{2}-q+1}-\beta_{1}^{\frac{k}{2}+q+1}\right|.
\end{align*}
Since $|\beta_1| < 1,$
\begin{align*}
\sum_{q=1}^{\frac{k-2}{2}}\left| \beta_1^{\frac{k}{2}-q+1}-\beta_{1}^{\frac{k}{2}+q+1}\right|
\leq&\left| \beta_1\right|^{\frac{k}{2}+1} \left( \sum_{q=1}^{\frac{k-2}{2}}  
\left|\beta_{1}^{-q}\right| +\sum_{q=1}^{\frac{k-2}{2}}\left| \beta_{1}^{q}\right| \right) \\
=& \left| \beta_{1}\right|^{\frac{k}{2}+1} \left( \frac{\left| \beta_{1}\right|^{1-\frac{k}{2}} -1}{1-\left| \beta_{1}\right|} +\frac{|\beta_{1}|-|\beta_{1}|^{\frac{k}{2}}}{1-|\beta_1|} \right) \\
=& \frac{|\beta_{1}|^{2}-|\beta_{1}|^{\frac{k}{2}+1}+|\beta_{1}|^{\frac{k}{2}+2}-|\beta_{1}|^{k+1}}{1-|\beta_{1}|}
<\infty. 
\end{align*}
Therefore,
\begin{align*} \label{3.7}
 \sum_{n=1}^{k-1}|b_{nk}|<\infty.
\end{align*}
Now,
\begin{align*}
\sum_{n=k}^{\infty}|b_{nk}|=&|b_{kk}|+|b_{k+2,k}|+|b_{k+4,k}|...\\
=&\sum_{p=0}^{\infty}|b_{k+2p,k}|
%=& \frac{1}{s_2(\beta_{1}^2-1)}\sum_{p=0}^{\infty} |\beta_{1}^{\frac{k+2p}{2}-\frac{k}{2}+1}-\beta_{1}^{\frac{k+2p}{2}+\frac{k}{2}+1}|
=\frac{1}{s_2(\beta_{1}^2-1)}
\sum_{p=0}^{\infty}|\beta_{1}^{p+1}-\beta_{1}^{p+k+1}|
%\leq & \frac{1}{s_2(\beta_{1}^2-1)} \left( \sum_{p=0}^{\infty}|\beta_{1}|^{p+1}+\sum_{p=0}^{\infty}|\beta_{1}|^{p+k+1}\right) \\
<\infty .
\end{align*}
Therefore, 
\begin{equation*}\label{3.8}
\sum_{n=k}^{\infty}|b_{nk}|<\infty.
\end{equation*}
Hence we have,
\begin{align}\label{3.9}
\sum_{n=1}^{\infty}|b_{nk}|<\infty,~~ \text{if $k$ is even.}
\end{align}
By using relations (\ref{3.6}) and (\ref{3.9}) we get,
\[\norm{(T_0-\lambda I)^{-1} }_{1}=  \sup\limits_{k} \sum\limits_{n=1}^{\infty}|b_{nk}| = \sup_k S_k < \infty,~~\text{for all $k$}. \] 
%		Hence,
%\[\norm{(T_0-\lambda I)^{-1}}_{(l_1,l_1)} <\infty \text{ for all} ~~k.\]\\
 As the matrix of inverse operator $(T_0-\lambda I)^{-1}$ is symmetric, performing similar calculations to the rows of $B = (b_{nk})$ we get, \[\norm{(T_0-\lambda I)^{-1}}_{\infty}<\infty.\]
 From Lemma \ref{L3.5}, it follows \[(T_0-\lambda I)^{-1}\in B(\ell_p).\]
 Hence it is proved that,
\begin{equation*}\label{1.11}
\sigma(T_0,\ell_p)\subseteq[r_1-2s_1,r_1+2s_1]\cup[r_2-2s_2,r_2+2s_2].
\end{equation*}
For the reverse inclusion relation let $\lambda\in[r_1-2s_1,r_1+2s_1]\cup[r_2-2s_2,r_2+2s_2]$ and $y=(1,1,0,0, \cdots)\in \ell_p$. Define $x = (x_1, x_2, x_3, \cdots)$ by
\[ x=(T_0-\lambda I)^{-1}y.\]
This implies, 
\begin{equation*}
x_n=\begin{cases}
\frac{-\beta_{1}^{\frac{n}{2}}}{s_2},\hspace{.5cm}\text{ if $n$ is even}\\
\frac{-\alpha_1^{\frac{n+1}{2}}}{s_1},\hspace{.5cm}\text{if $n$ is odd}.
\end{cases}
\end{equation*}
Also $\lambda\in [r_1-2s_1,r_1+2s_1]$ implies $|\alpha_1|=1$ which also implies $\frac{-\alpha_1^{\frac{n+1}{2}}}{s_1}\not\to 0$ as $n \to \infty$. Similarly $\lambda\in[r_2-2s_2,r_2+2s_2]$ implies $\frac{-\beta_{1}^{\frac{n}{2}}}{s_2} \not\to0$ as $n \to \infty$. Hence $x \not \in\ell_p$ and $\lambda \in \sigma(T_0,\ell_p)$. This proves the following relation
\begin{equation*}\label{1.13}
[r_1-2s_1,r_1+2s_1]\cup[r_2-2s_
2,r_2+2s_2] \subseteq \sigma(T_0,\ell_p).
\end{equation*}  
Hence, we conclude that
\[\sigma(T_0,\ell_p)=[r_1-2s_1,r_1+2s_1]\cup[r_2-2s_2,r_2+2s_2].\]
\end{proof}

\begin{corollary} 
The continuous spectrum of $T_0$ over $\ell_{p}$ is given by
	$$\sigma_c(T_0,\ell_p)=\left[r_{1}-2s_1, r_1+2s_1\right] \cup [r_2-2s_2, r_2+2s_2].$$
	\end{corollary}
	\begin{proof}
		It is evident that $\sigma(T_0,\ell_{p})$ is the disjoint union of $\sigma_p(T_0,\ell_{p})$,  $\sigma_r(T_0,\ell_{p})$ and  $\sigma_c(T_0,\ell_{p})$, we have 
		\begin{eqnarray*} 
			\sigma(T_0,\ell_p)=\sigma_c(T_0,\ell_p).
		\end{eqnarray*}
		Hence, $ \sigma_c(T_0,\ell_p)=\left[r_{1}-2s_1, r_1+2s_1\right] \cup [r_2-2s_2, r_2+2s_2]$.
	\end{proof}
	
\begin{corollary}\label{ess}
Essential spectrum of $T_0$ defined over $\ell_{p}$ is given by $$\sigma_{ess}(T_0,\ell_{p})=[r_1-2s_1,r_1+2s_1]\cup[r_2-2s_2,r_2+2s_2].$$
\end{corollary}	
\begin{proof}
It is well-known that $\sigma_c(T_0, \ell_p) \subseteq \sigma_{ess}(T_0,\ell_p)$ and we have $$\sigma_{ess}(T_0,\ell_p)\subseteq \sigma(T_0,\ell_{p}) = \sigma_c(T_0,\ell_{p})\subseteq \sigma_{ess}(T_0, \ell_{p}).$$ Hence, the desired result is obvious.
\end{proof}	
Using the relations which are mentioned in Proposition \ref{p3.1} we can easily obtain the following results.
\begin{corollary}
The compression spectrum, approximate point spectrum and defect spectrum of $T_0$ over $\ell_p$ are as follows
\begin{enumerate}
\item[(i)] $\sigma_{co}(T_0,\ell_p)=\emptyset$,
\item[(ii)]$\sigma_{app}(T_0,\ell_p)=\left[r_{1}-2s_1, r_1+2s_1\right] \cup [r_2-2s_2, r_2+2s_2]$,
\item[(iii)] $\sigma_{\delta}(T_0,\ell_p)=\left[r_{1}-2s_1, r_1+2s_1\right]\cup[r_2-2s_2, r_2+2s_2]$. 
\end{enumerate}
\end{corollary}

In particular, if $r_1 = r_2 = r$ and $s_1 = s_2 = s$ then the operator $T_0$ reduces to an operator with the following matrix representation
\[ \begin{pmatrix}
			r & 0 &s&0&0&\cdots &  \\
			0 & r & 0& s&0& \cdots&\\
			s&0&r&0&s&\cdots&\\
			0&s&0&r&0&\cdots&\\
			0&0&s&0&r&\cdots\\
			\vdots  & \vdots  & \vdots &\vdots&\vdots& \ddots
		\end{pmatrix}.
\]
The following results can be obtained from the previously proved results in this setting
\begin{corollary} \label{cor1}
The spectrum and other spectral subdivisions over $\ell_{p}$ are given by,
\begin{enumerate}
\item[(i)] $\sigma(T_0,\ell_p) = \sigma_{c}(T_0,\ell_p) = \sigma_{app}(T_0,\ell_p) = \sigma_{\delta}(T_0,\ell_p) = \left[r-2s, r+2s\right],$
\item[(ii)] $\sigma_p(T_0,\ell_p) = \sigma_{r}(T_0,\ell_p) = \sigma_p(T_0^*,\ell_p^*) = \emptyset.$
\end{enumerate}
\end{corollary}
 It is interesting to note that the spectrum and fine spectra mentioned in Corollary \ref{cor1} coincide with the spectrum and fine spectrum of the tridiagonal matrix $U(s,r,s)$ defined over $\ell_p$ which is obtained in \cite{karakaya2012fine}.

\section{\bf Spectra of $T = T_0 + K$}

In this section, we focus on the spectral properties of the operator $T$ defined over $\ell_p$ which can be expressed as $T=T_0+K$ where $K$ is represented by the following matrix
\begin{equation*}
	K=
	\begin{pmatrix}
		a_1-r_1 & 0 &b_1-s_1&0&0&\cdots &  \\
		0 & a_2-r_2 & 0& b_2-s_2&0& \cdots&\\
		c_1-s_1&0&a_3-r_1&0&b_3-s_1&\cdots&\\
		0&c_2-s_2&0&a_4-r_2&0&\cdots&\\
		0&0&c_3-s_1&0&a_5-r_1&\cdots\\
		\vdots  & \vdots  & \vdots &\vdots&\vdots& \ddots
	\end{pmatrix}.
\end{equation*}
The following result proves the compactness of $K$ on $\ell_{p}$.

\begin{theorem}\label{theorem1.6}
  The operator $K$ is a compact operator on $\ell_{p}$.
 
\end{theorem}
\begin{proof}
	The operator $K$ on $\ell_{p}$ can be represented by the following infinite matrix
	 \begin{equation*}
		K=
		\begin{pmatrix}
			u_1 & 0 &v_1&0&0&\cdots &  \\
			0 & u_2 & 0& v_2&0& \cdots&\\
			w_1&0&u_3&0&v_3&\cdots&\\
			0&w_2&0&u_4&0&\cdots&\\
			\vdots  & \vdots  & \vdots &\vdots&\vdots& \ddots
		\end{pmatrix},
	\end{equation*}
where $\{u_n\}$, $\{v_n\}$ and $\{w_n\}$ are null sequences, which are defined as follows,
 
\begin{equation*}
	u_n=\begin{cases*}
		a_n-r_1,~~ \mbox{$n$ is odd}\\
		a_n-r_2, ~~ \mbox{$n$ is even},
	\end{cases*}~~~~\qquad
	v_n=\begin{cases*}
		b_n-s_1,~~ \mbox{ $n$ is odd}\\
		b_n-s_2,~~ \mbox{$n$ is even}
	\end{cases*}
\end{equation*}
and
\begin{equation*}
w_n=\begin{cases*}
	c_n-s_1,~~ \mbox{ $n$ is odd}\\
	c_n-s_2,~~ \mbox{$n$ is even}.
\end{cases*}
\end{equation*}
Let $x=\{x_1,x_2,x_3...\}\in \ell_p$. 
%\begin{equation*}
%Ax=\left\lbrace {u_1x_1+v_1x_3,u_2x_2+v_2x_4,w_1x_1+u_3x_3+v_3x_5,w_2x_2+u_4x_4+v_4x_6...w_nx_n+u_{n+2}x_{n+2}+v_{n+2}x_{n+4}...}\right\rbrace 
%\end{equation*}
 We construct a sequence of compact operators $\left\lbrace K_n\right\rbrace $ such that for $i \in \mathbb{N}$, \\
\begin{equation*}
(K_n(x))_i=\begin{cases}
	(Kx)_i,~~ i=1,2,...n\\
	0,~~ \text{otherwise.}
\end{cases}	
\end{equation*}
%This implies
%\begin{equation*}
%	(A-A_n)(x_i)=\begin{cases}
	%	0,~~ i=1,2,\cdots n\\
	%	w_ix_i+u_{i+2}x_{i+2}+v_{i+2}x_{i+4},~~ \text{for all} ~~i\geq n+1.
	%\end{cases}
%\end{equation*}
For $n \geq2$,
\begin{align*}
\norm{(K-K_n)x}_{p}
%=&\left( \sum_{i=1}^{\infty}\left( |w_{n+i}x_{n+i}+u_{n+i+2}x_{n+i+2}+v_{n+i+2}x_{n+i+4}|\right)^p\right)^{\frac{1}{p}}  \\
= & \left( \sum_{k=n-1}^{\infty}|w_kx_k+u_{k+2}x_{k+2}+v_{k+2}x_{k+4}|^p\right) ^{\frac{1}{p}}\\
%\leq & \left( \sum_{n=1}^{\infty}\left(|w_nx_n|+|u_{n+2}x_{n+2}|+|v_{n+2}x_{n+4}|\right) ^p \right) ^{\frac{1}{p}}\\
%\leq & \left( \sum_{n=1}^{\infty}|w_nx_n|^p\right) ^ \frac{1}{p}+\left( \sum_{n=1}^{\infty}|u_{n+2}x_{n+2}|^p\right) ^{\frac{1}{p}}+ \left( \sum_{n=1}^{\infty}|v_{n+2}x_{n+4}|^p\right) ^{\frac{1}{p}}\\
\leq & \left( \sup_{k\geq n-1}|w_k|\right)\norm{x}_p + \left( \sup_{k \geq n-1}|u_k|\right)\norm{x}_p+ \left( \sup_{k\geq n-1}|v_k|\right) \norm{x}_p.
\end{align*} 
This implies,
$$ \|{K-K_n}\|_p ~\leq  \sup_{k \geq n-1}|w_k|+\sup_{k \geq n-1 }|u_k|+\sup_{k \geq n-1}|v_k|.$$
Thus, $\{K_n\}$ converges to $K$ as $n\to\infty$ in operator norm and hence $K$ is a compact operator over $\ell_p$.
\end{proof}
Hence, the operator $T$ is a compact perturbation of $T_0$ and since $T_0 \in B(\ell_p)$, $K$ is compact on $\ell_p,$ the operator $T$ is an bounded linear operator on $\ell_p.$

Next we derive an inclusion relation between $\sigma(T_0, \ell_{p})$ and $\sigma(T,\ell_{p})$ and to prove this result we require the following Lemma.
\begin{lemma}\cite[p.373]{gohberg2013classes} \label{lemma4}
	Let $T : X\rightarrow X$ be an operator with a non-empty resolvent set, and let $\Omega$	be an open connected subset of $\mathbb{C}\setminus \sigma_{ess}(T)$. If $\Omega \cap \rho(T) \neq \emptyset$ then $\sigma(T)\cap \Omega$
	is a finite or countable set, with no accumulation point in $\Omega$, consisting of eigenvalues of $T$ of finite type.
\end{lemma} 
\begin{theorem}\label{t4.11}
The spectrum of $T$ over $\ell_{p}$ satisfies the following inclusion relation $$\sigma(T_0,\ell_p) \subseteq \sigma(T,\ell_p)$$ and $\sigma(T,\ell_p)\setminus \sigma(T_0,\ell_p)$ contains finite or countable number of eigenvalues of $T$ of finite type with no accumulation point in $\sigma(T,\ell_p)\setminus \sigma(T_0,\ell_p)$.  
\end{theorem}
\begin{proof}
Suppose $\lambda \not \in\sigma(T, \ell_p)=\sigma(T_0+K, \ell_p)$. This implies $(T_0+K-\lambda I)^{-1}$ exists and belongs to $B(\ell_p)$. Then there exists $U \in B(\ell_{p})$ such that $(T_0+K-\lambda I)U=I$. Hence,
\begin{equation}\label{T_0}
 KU-I=-(T_0-\lambda I)U
\end{equation} 
 and  $(KU-I)x=0$ implies $(T_0-\lambda I)Ux=0$. This gives us $Ux \in N(T_0-\lambda I)=\{0\}$ as $\sigma(T_0, \ell_{p})=\emptyset$. Therefore $x=0$ and consequently $1 \not \in \sigma_p(KU)$.
As $K$ is compact operator, $KU$ is also a compact operator and it follows that $1 \not \in \sigma(KU)$. Hence, $(KU-I)$ is invertible and consequently $(T_0-\lambda I)$ is invertible by equation (\ref{T_0}). This implies, $\lambda \not \in \sigma(T_0)$ and $\sigma(T_0)\subseteq \sigma(T)$.
\\
For the second part we have,  $\rho(T)$ is non-empty as $T$ is a bounded linear operator and $\sigma_{ess}(T,\ell_p)=\sigma_{ess}(T_0,\ell_p)$. Let $\Omega=\mathbb{C}\setminus \sigma(T_0,\ell_p)$ then $\Omega \cap \rho(T,\ell_p)\neq \emptyset$. The set $\Omega$ is open connected subset of $\mathbb{C}\setminus \sigma_{ess}(T_0,\ell_p) = \mathbb{C}\setminus \sigma_{ess}(T,\ell_p)$. Then by using Lemma \ref{lemma4}  we have, $\sigma(T,\ell_p)\cap \Omega$ is a finite or countable set with no accumulation point in $\Omega$ consisting eigenvalues of finite type. 
\end{proof}

\begin{corollary}\label{ess1}
	$\sigma_{ess}(T, \ell_{p})=\sigma(T_0, \ell_{p})=[r_1-2s_1,r_1+2s_1]\cup[r_2-2s_2,r_2+2s_2]$.
\end{corollary}
\begin{proof}
	Since a compact perturbation does not effect the Fredholmness and index of a Fredholm operator, it follows $\sigma_{ess}(T_0,\ell_{p})=\sigma_{ess}(T,\ell_{p})$. Hence, by using Corollary \ref{ess}, we have 
	\begin{equation*}
		\sigma_{ess}(T, \ell_{p})=\sigma(T_0, \ell_{p})=[r_1-2s_1,r_1+2s_1]\cup[r_2-2s_2,r_2+2s_2].
	\end{equation*}
\end{proof}
We now focus on the point spectrum of $T$ on $\ell_{p}$. First we analyze the eigenvalues of $T$ lying in $\sigma_{ess}(T, \ell_{p})=\sigma(T_0, \ell_p)$, in particular we derive sufficient conditions for the absence of point spectrum on $\sigma_{ess}(T, \ell_p)$. In Theorem \ref{poT}, sufficient conditions are provided in terms of the rate of convergence of the sequences $\{a_{2n-1}\}, \ \{a_{2n}\},  \ \{b_{2n-1}\}, \ \{b_{2n}\}, \ \{c_{2n-1}\}$ and $\{c_{2n}\}$. Sufficient conditions of absence of point spectrum on $\sigma_{ess}(T, \ell_{p})$ are also provided in Theorem \ref{sc} in terms of the entries of the matrix $T$.

\begin{theorem}\label{poT}
If the convergence of the sequences $\{a_{2n-1}\}, \ \{a_{2n}\}, \ \{b_{2n-1}\}, \ \{b_{2n}\}, \\ \{c_{2n-1}\}$ and $\{c_{2n}\}$ are exponentially fast then $$\sigma_{ess}(T, \ell_{p}) \cap \sigma_{p}(T,\ell_{p})= \emptyset.$$
\end{theorem}
\begin{proof}
	Let $\lambda \in \sigma_{ess}(T, \ell_{p})$.
 The equation $Tx= \lambda x$ for some $\lambda \in \mathbb{C} $ reduces to the following system
\begin{eqnarray*}
	\begin{rcases}
     a_1x_1+b_1x_3&=\lambda x_1\\
	 a_2x_2+b_2x_4&=\lambda x_2\\
	 c_1x_1+a_3x_3+b_3x_5&=\lambda x_3\\
	 c_2x_2+a_4x_4+b_4x_6&=\lambda x_4\\
                	 &\vdots 
     \end{rcases}  .         	
\end{eqnarray*}
If we separate the odd and even terms of the sequences $\{a_n\},$ $\{b_n\}$ and $\{c_n\}$, the above system of equations can also be expressed as
\begin{equation} \label{point1}
\begin{rcases}
c_{2n-1}x_{2n-1}+(a_{2n+1}-\lambda) x_{2n+1}+b_{2n+1}x_{2n+3}=0,\\
c_{2n}x_{2n}+(a_{2n+2}- \lambda)x_{2n+2}+b_{2n+2}x_{2n+4}=0,
\end{rcases}
\end{equation}
where $n \in \mathbb{N}$ with the initial conditions
\begin{equation} \label{point2}
\begin{rcases}
a_1x_1+b_1x_3=\lambda x_1,\\
a_2x_2+b_2x_4=\lambda x_2.
\end{rcases}
\end{equation}
Introducing two sequences $\{y_n\}$ and $\{z_n\}$ such that $y_n = x_{2n-1}$ and $z_n = x_{2n}$ for $n \in \mathbb{N},$ the system (\ref{point1}) with the initial conditions (\ref{point2}) reduces to
\begin{equation} \label{point_diff}
c_{2n-1}y_n+(a_{2n+1}-\lambda) y_{n+1}+b_{2n+1}y_{n+2}=0,
\end{equation}
\begin{equation}\label{point_diff1}
	c_{2n}z_n+(a_{2n+2}- \lambda)z_{n+1}+b_{2n+2}z_{n+2}=0,
\end{equation}
where $n \in \mathbb{N} \cup \{0\}$ with $y_0 = z_0 = 0.$ Using the assumed convergence of the sequences $\{a_{2n-1}\}$, $\{a_{2n}\},$ $\{b_{2n-1}\},  \{b_{2n}\}$, $\{c_{2n-1}\}$, and $\{c_{2n}\},$ the characteristic polynomials of the difference equations (\ref{point_diff}) and (\ref{point_diff1}) are
 \begin{eqnarray}
&& t^2+p_1t+1=0 ,\label{point_poly1} \\
&& t^2+p_2t+1=0,  \label{point_poly2}
 \end{eqnarray}
where $p_1 = \frac{r_1 - \lambda}{s_1}$ and $p_2= \frac{r_2 - \lambda}{s_2}$. Let $\mu_1, \ \mu_2$ and $\gamma_1, \ \gamma_2$ are the pair of roots of equations (\ref{point_poly1}) and (\ref{point_poly2}) respectively. Also we have $$\lambda \in \sigma_{ess}(T,\ell_p)=\sigma(T_0,\ell_p)=[r_1-2s_1,r_1+2s_1]\cup[r_2-2s_2,r_2+2s_2].$$
 Let $\lambda \in [r_1-2s_1,r_1+2s_1]$ i.e., $p_1 \in [-2,2]$. Then the roots $\mu_1$, $\mu_2$ satisfies $|\mu_1|=1$ and $|\mu_2|=1$.
Since the convergence of $\{a_{2n-1}\}, \{b_{2n-1}\}, \{c_{2n-1}\}$ are exponentially fast, if $\{y_n\}$ is a solution of equation   (\ref{point_diff}) then by Theorem 2.3 \cite{agarwal2007asymptotic}, it can be deduced that, either $y_n=0$ for large $n$ or there exists $\rho \in (0,1)$ such that 
\begin{equation}\label{t4.8}
	y_n=y_n'+O((1-\rho)^n), ~~ \quad \mbox{for large $n$} 
\end{equation}
where $\{y_n'\}$ is a solution of limiting equation of (\ref{point_diff}) which is given by $$y_{n}+p_1y_{n+1}+y_{n+2}=0,$$
and the solution of this limiting equation is already obtained in Case 1, Case 2, Case 3 of Theorem \ref{t1.2}.
Based on equation (\ref{t4.8}), there exists an $M>0$ such that 
\begin{equation*}
	|y_n-y_n'|\leq M(1-\rho)^n.
\end{equation*}
Thus,
\begin{equation*}
|y_n'|^p\leq(|y_n|+M(1-\rho)^n)^p.
\end{equation*}
Applying Jensen's inequality we get,
\begin{equation*}
	|y_n'|^p\leq 2^{(p-1)}(|y_n|^p+M^p(1-\rho)^{np}).
\end{equation*}
As $|\mu_1|=1$ and $|\mu_2|=1$, from Theorem \ref{t1.2} we have $\{y_n'\} \not \in \ell_{p}$. Also $0< 1- \rho <1$ implies, $\{y_n\} \not \in \ell_{p}$. Hence, $\lambda \not \in \sigma_p(T,\ell_{p})$.
In a similar way, if $\lambda \in [r_2-2s_2,r_2+2s_2]$ we can obtain that 
$\lambda \not\in \sigma_p(T,\ell_{p})$.
Hence the desired result is proved.
\end{proof}

In the next theorem, we apply transfer matrix approach as discussed in \cite{janas2003similarity, janas1998point} . This enables us to examine the sufficient condition for the absence of point spectrum in essential spectrum of $T$ in terms of the entries of matrix $T$.
\begin{theorem}\label{sc}
	If \(\lambda \in \sigma_{ess}(T,\ell_{p})\) satisfies either of the following conditions\\
	(i) $ \sum_{n=1}^{\infty} \prod_{j=1}^n \left[\frac{1}{2} \left(P_j(\lambda)-\sqrt{P_j(\lambda)^2-\left| \frac{2c_{2j-1}}{b_{2j+1}}\right|^2}\right) \right]^{\frac{p}{2}}=+\infty$ \\
	\\ or\\
	(ii) $ \sum_{n=1}^{\infty} \prod_{j=1}^n \left[\frac{1}{2} \left( Q_j(\lambda)-\sqrt{Q_j(\lambda)^2-\left| \frac{2c_{2j}}{b_{2j+2}}\right|^2}\right) \right]^{\frac{p}{2}}=+\infty$\\
	where,
	$$  P_j(\lambda)= \left| \frac{c_{2j-1}}{b_{2j+1}}\right| ^2 +\left| \frac{a_{2j+1}-\lambda}{b_{2j+1}}\right|^2+1,~~Q_j(\lambda)= \left| \frac{c_{2j}}{b_{2j+2}}\right| ^2 +\left| \frac{a_{2j+2}-\lambda}{b_{2j+2}}\right|^2+1, $$
	then \(\lambda \notin \sigma_{p}(T,\ell_{p})\).
\end{theorem}
	\begin{proof}
		
	(i) Let $\lambda \in \sigma_{ess}(T, \ell_{p})= \sigma(T_0, \ell_{p}).$	Using \(Tx = \lambda x\), we have the following difference equation 
		\begin{equation} \label{point11}
			\begin{rcases}
				c_{2n-1}x_{2n-1}+(a_{2n+1}-\lambda) x_{2n+1}+b_{2n+1}x_{2n+3}=0\\
				c_{2n}x_{2n}+(a_{2n+2}- \lambda)x_{2n+2}+b_{2n+2}x_{2n+4}=0,
			\end{rcases}
		\end{equation}
		where $n \in \mathbb{N}$ with the initial conditions
		\begin{equation} \label{point22}
			\begin{rcases}
				a_1x_1+b_1x_3=\lambda x_1,\\
				a_2x_2+b_2x_4=\lambda x_2.
			\end{rcases}
		\end{equation}
		 The first equation of (\ref{point11}) can be written in the following form \\
		\begin{equation*}
			\left(\begin{array}{l}
				x_{2 n+1} \\
				x_{2 n+3}
			\end{array}\right)=B_n(\lambda)\left(\begin{array}{l}
				x_{2 n-1} \\
				x_{2 n+1}
			\end{array}\right), ~~n \in \mathbb{N}\cup \{0\}
		\end{equation*}
		where $$
		B_n(\lambda)=\left(\begin{array}{lc}
			0 & 1 \\
			\frac{-c_{2 n-1}}{b_{2 n+1}} & \frac{-(a_{2 n+1}-\lambda)}{b_{2 n+1}}
		\end{array}\right),
		$$
	with  $x_{-1}=0$ when $n=0$. This includes the initial condition $$a_1x_1+b_1x_3= \lambda x_1.$$
In this setting, we have
		\begin{equation}\label{b1}
			\left(\begin{array}{l}
				x_{2 n+1} \\
				x_{2 n+3}
			\end{array}\right)=B_n(\lambda) B_{n-1}(\lambda) \cdots B_{1}(\lambda)y
		\end{equation}
		where, $
		y=\left(\begin{array}{l}
			x_1 \\
			\frac{(\lambda-a_1)}{b_1}x_1
		\end{array}\right)$.
	Also		
		\begin{eqnarray} \label{eqnarray}
		\left\|B_n B_{n-1} \cdots B_1 y\right\|_p^p & \geq & \max \left\{2^{\frac{1}{p}-\frac{1}{2}}, 1\right\} \left(\left\|B_n \cdots B_1 y\right\|_2^2\right)^{p / 2} \nonumber \\
			&=&\max \left\{2^{\frac{1}{p}-\frac{1}{2}}, 1\right\}(\left\langle B_n \cdots B_1 y, B_n \cdots B_1 y\right\rangle)^{\frac{p}{2}} \nonumber \\
			&=& M_1 \left(\frac{\left\langle B_n \cdots B_1 y, B_n \cdots B_1 y\right\rangle}{\|y\|^2}\right)^{\frac{p}{2}},
		\end{eqnarray}
where $M_1=\|y\|^p \max \left\{2^{\frac{1}{p}-\frac{1}{2}}, 1\right\}.$ Using the singular value analog of the famous Courant–Fischer theorem \cite[Theorem 7.3.8]{horn2012matrix} and the  result 
\[\sigma_{{min}} (AB) \geq \sigma_{{min}}(A) \  \sigma_{{min}}(B)\] 
where $\sigma_{{min}} (A)$ denotes the smallest singular value of a matrix $A$, it follows from the relation (\ref{eqnarray}) that  
		\begin{align*}
				\left\|B_n B_{n-1} \cdots B_1 y\right\|_p^p 
			&\geq M_1 \sigma_{min}^p\left(B_n B_{n-1} \ldots B_1\right) \\
			&\geq M_1 \sigma_{min}^p\left(B_n\right) \sigma_{min}^p\left(B_{n-1}\right) \ldots \sigma_{min}^ p\left(B_1\right), ~~n \in \mathbb{N}.
		\end{align*}
		From equation (\ref{b1}) we obtain that
		\begin{equation*}
			\left|x_{2 n+1}\right|^p+\left|x_{2 n+3}\right|^p \geq M_1 \sigma_{min}^p\left(B_n\right) \sigma_{min}^p\left(B_{n-1}\right) \cdots \sigma_{min}^p\left(B_1\right).
		\end{equation*}
	Taking summation over \(n\), the above relation reduces to 
		\begin{equation*}
			2\left[\sum_{n=1}^{\infty}\left|x_{2 n+1}\right|^p\right] \geq M_1 \sum_{n=1}^{\infty} \prod_{j=1}^n \sigma_{min}^p\left(B_j\right).
		\end{equation*}
	This implies, 
	\begin{equation} \label{eq_singular}
		\left[\sum_{n=1}^{\infty}\left|x_{2 n+1}\right|^p\right] \geq M_1^\prime \sum_{n=1}^{\infty} \prod_{j=1}^n \sigma_{min}^p\left(B_j\right)
	\end{equation}
		for some constant $M_1^\prime > 0.$ The lowest singular value of $B_j$ is given by
{\footnotesize \[\sigma_{min}(B_j)= \left[\frac{1}{2} \left( \left| \frac{c_{2j-1}}{b_{2j+1}}\right| ^2 +\left| \frac{a_{2j+1}-\lambda}{b_{2j+1}}\right|^2+1-\sqrt{\left(\left| \frac{c_{2j-1}}{b_{2j+1}}\right| ^2 +\left| \frac{a_{2j+1}-\lambda}{b_{2j+1}}\right|^2+1\right)^2-\left| \frac{2c_{2j-1}}{b_{2j+1}}\right|^2}\right) \right]^{\frac{1}{2}}\]}		
where we assume $\sqrt{a^2} = a$ for some positive $a$. Hence, if
$$\sum_{n=1}^{\infty} \prod_{j=1}^n \left[\frac{1}{2} \left(P_j(\lambda)-\sqrt{P_j(\lambda)^2-\left| \frac{2c_{2j-1}}{b_{2j+1}}\right|^2}\right) \right]^{\frac{p}{2}}=+\infty$$	
where, 	
$$P_j(\lambda)= \left| \frac{c_{2j-1}}{b_{2j+1}}\right| ^2 +\left| \frac{a_{2j+1}-\lambda}{b_{2j+1}}\right|^2+1,$$
then \(\{x_n\} \notin \ell_p\) and consequently \(\lambda \notin \sigma_p(T,\ell_p)\). This proves the first part of the result.
		
(ii) As similar to part (i) the second equation of (\ref{point11}) can be written as 
		\begin{equation*}
			\left(\begin{array}{l}
				x_{2 n+2} \\
				x_{2 n+4}
			\end{array}\right)=C_n(\lambda)\left( \begin{array}{l}
				x_{2 n} \\
				x_{2 n+2}
			\end{array}\right), ~~n \in \mathbb{N}
		\end{equation*}
		where $$
		C_n(\lambda)=\left(\begin{array}{lc}
			0 & 1 \\
			\frac{-c_{2 n}}{b_{2 n+2}} & \frac{-(a_{2 n+2}-\lambda)}{b_{2 n+2}}
		\end{array}\right),
		$$
		with $x_{0}=0$ when $n=0$. This includes the initial condition $$a_2x_2+b_2x_4= \lambda x_2.$$
	 We can obtain the desired result by using the same argument as in first part.
	\end{proof}
	
\begin{remark}
Instead of calculating the singular value $\sigma_{min} (B_j)$ in the relation (\ref{eq_singular}), various lower bounds for the same can be used to obtain a less complicated expression than $\sigma_{min} (B_j).$ Several researchers have been working to refine the lower bound of lowest singular value. Some of the recent works for the lower bound of smallest singular value of a matrix can be found in \cite{singular2,singular3,singular4,singular5,singular6}.
\end{remark}

Now we focus our study on the point spectrum of $T$.
Under the sufficient conditions as mentioned in previous two results, we have
$\sigma_p(T, \ell_{p})\cap \sigma(T_0, \ell_{p})=\emptyset$. In this case, all the eigenvalues of $T$ are lying outside the set $\sigma(T_0, \ell_p)$.
To characterize the eigenvalues, let $Tx=\lambda x$, $x \in \mathbb{C^N}$ and $\lambda \in \sigma(T_0,\ell_{p})^c$ where $\sigma(T_0,\ell_{p})^c$ denotes the complement of $\sigma(T_0,\ell_{p})$. From equations (\ref{point_diff}) and (\ref{point_diff1}) in Theorem \ref{poT}, we have the following system 
 \begin{equation} \label{point_diff4}
 	c_{2n-1}y_n+(a_{2n+1}-\lambda) y_{n+1}+b_{2n+1}y_{n+2}=0,
 \end{equation}
 \begin{equation}\label{point_diff5}
 	c_{2n}z_n+(a_{2n+2}- \lambda)z_{n+1}+b_{2n+2}z_{n+2}=0,
 \end{equation}
where $n \in \mathbb{N}\cup\{0\}$ with $y_0=z_0=0$ and $y_n=x_{2n-1}$, $z_n=x_{2n}$.
Clearly each of the difference equations (\ref{point_diff4}) and (\ref{point_diff5}) have two fundamental solutions. Let $\{y_n^{(1)}(\lambda),y_n^{(2)}(\lambda)\}$ and $\{z_n^{(1)}(\lambda),z_n^{(2)}(\lambda)\}$ are the sets of fundamental solutions of the equations  (\ref{point_diff4}) and (\ref{point_diff5}) respectively. Under this setting we have the following result.
\begin{theorem}\label{t4.12}
If either of the sufficient conditions mentioned in Theorem \ref{poT} and Theorem \ref{sc} hold true then the point spectrum of $T$ over $\ell_{p}$ is given by
\[\sigma_p(T,\ell_p)=\left\lbrace \lambda \in \mathbb{C}: y_{0}^{(1)}(\lambda) =0\right\rbrace \cup \left\lbrace \lambda \in \mathbb{C}:{z_0}^{(1)}(\lambda)=0\right\rbrace .\]
\end{theorem}

\begin{proof}
As $\sigma_{ess}(T, \ell_p) \cap \sigma_{p}(T, \ell_p) = \emptyset$, we restrict our search for point spectrum outside the set $[r_1-2s_1,r_1+2s_1] \cup [r_2-2s_2,r_2+2s_2].$
Let $\mu_1, \ \mu_2$ and $\gamma_1, \ \gamma_2$ are the pair of roots of equations (\ref{point_poly1}) and (\ref{point_poly2}) respectively which are the characteristic polynomials of equations (\ref{point_diff4}) and (\ref{point_diff5}) respectively.
 %We are going to prove our result by using the asymptotic behaviour of the general solutions of the difference equations (\ref{diff2}) and (\ref{diff3}). %We assume that the convergence of the sequences $\{a_{2n}\}, \ \{a_{2n-1}\}, \ \{b_{2n}\}, \ \{b_{2n-1}\}, \ \{c_{2n}\},$ and $\{c_{2n-1}\}$ are exponentially fast. Consider the following cases for this purpose.
Since $\lambda \notin [r_1-2s_1,r_1+2s_1] \cup [r_2-2s_2,r_2+2s_2]$, we have $p_1 \not \in [-2,2]$, and without loss of generality we assume that $|\mu_1|<1$ and $|\mu_2|>1$. By Perron's First Theorem \cite[p.344]{difference} it can be deduced that
$$\lim_{n\to\infty} \frac{y_{n+1}^{(1)}(\lambda)}{y_{n}^{(1)}(\lambda)}=\mu_1,\quad \lim_{n\to\infty}\frac{y_{n+1}^{(2)}(\lambda)}{y_{n}^{(2)}(\lambda)} =\mu_2.$$
Hence $\{y_n^{(1)}(\lambda)\}\in \ell_p$ but $\{y_n^{(2)}(\lambda)\}\not \in \ell_p$ and the general solution of the difference equation (\ref{point_diff4}), which is the linear combination of the fundamental solutions, is given by
\begin{equation*}
	y_n(\lambda)=c_1y_{n}^{(1)}(\lambda)+c_2y_{n}^{(2)}(\lambda),\text{$n \in \mathbb{N}\cup\{0\}$}
\end{equation*}
where $c_1$ and $c_2$ are arbitrary constants.
In a similar way, we can assume that $|\gamma_1|<1$ and $|\gamma_2|>1$ and by Perron's First Theorem we have
$$\lim_{n\to\infty} \frac{z_{n+1}^{(1)}(\lambda)}{z_{n}^{(1)}(\lambda)}=\gamma_1,\quad\lim_{n\to\infty} \frac{z_{n+1}^{(2)}(\lambda)}{z_{(n)}^{(2)}(\lambda)}=\gamma_2.$$\\
This implies $\{z_{n}^{(1)}(\lambda)\}\in \ell_p$ and $\{z_{n}^{(2)}(\lambda)\}\not \in \ell_p$
 and the general solution of the difference equation (\ref{point_diff5}) is given by
\begin{equation*}
	z_{n}(\lambda)=d_1{z_{n}^{(1)}(\lambda)}+d_2z_{n}^{(2)}(\lambda), \text{$n \in \mathbb{N}\cup\{0\}$}
\end{equation*}
where $d_1$ and $d_2$ are arbitrary constants and consequently the general solution of the system $Tx= \lambda x$ is given by $x_n(\lambda)$ where
\begin{eqnarray*}
x_{2n-1}(\lambda)=c_1y_{n}^{(1)}(\lambda)+c_2y_{n}^{(2)}(\lambda),\text{$n \in \mathbb{N}$}\\
x_{2n}(\lambda)=d_1{z_{n}^{(1)}(\lambda)}+d_2z_{n}^{(2)}(\lambda),\text{$n \in \mathbb{N}$}
\end{eqnarray*}
with $x_{-1}(\lambda)=x_0(\lambda)=0$. Consider $$S_1=\left\lbrace \lambda \in \mathbb{C}: y_0^{(1)}(\lambda)=0\right\rbrace \cup \left\lbrace \lambda \in \mathbb{C} :z_0^{(1)}(\lambda)=0\right\rbrace. $$
Let $\lambda \in S_1$, then ${y_0^{(1)}({\lambda})}=0$ or ${z_0^{(1)}(\lambda)=0}$.  If $y_0^{(1)}(\lambda)=0$, we can construct a non-trivial solution $x_n(\lambda)$ of the system $Tx= \lambda x$ in the following way. \\
Let $c_2=0$ and $d_1=d_2=0$. In this case we have  $y_n(\lambda)=y_n^{(1)}(\lambda)$ and $z_{n}(\lambda)=0$ for all $n$. 
%In this case $x_n(\lambda)\in \ell_{p}$. If $z_0^1(\lambda)=0$ then we can take $z_n(\lambda)=z_n^{(1)}(\lambda)\in \ell_p$ and $y_n(\lambda)=0$ for all $n$. %This satisfies the initial condition $y_0^1(\lambda)=0$ and $z_0^1(\lambda)=0$.
Since, $\{y_n^{(1)}(\lambda)\} \in \ell_{p}$, we have $x_n(\lambda)$ is a non-trivial solution of $Tx= \lambda x$ and $\{x_n(\lambda)\}\in \ell_{p}.$ 
If $z_0^{(1)}(\lambda)=0$ then in a similar way we can construct a non-trivial solution $x_n(\lambda)$ of $Tx=\lambda x$ where $x_n(\lambda)=0$, if $n$ is odd and $x_n(\lambda)= z_n^{(1)}(\lambda)$, if $n$ is even. Hence, $\lambda \in \sigma_p(T, \ell_{p})$ and consequently $S_1 \subseteq \sigma(T, \ell_{p})$. 
Now, suppose $\lambda \not \in S.$ Then $y_0^{(1)}(\lambda)\neq 0$ and $z_0^{(1)}(\lambda)\neq 0$.
Clearly $\lambda \in \sigma_p(T,\ell_p)$ if and only if $c_2=0$ and $d_2=0$.
Now we consider following cases with the assumption $c_2=d_2=0$.\\
\noindent Case 1: If $c_2=d_2=0$ and $c_1=0$, we have 
\[y_n(\lambda)=0~ \mbox{$\forall n$ and}~ z_n(\lambda)=d_1z_n^{(1)}(\lambda)~ \forall n .\] Using the initial condition $z_0(\lambda)=0$, we have 
$d_1z_0^{(1)}(\lambda)=0.$
If $d_1=0$ then we get a trivial solution and if $d_1 \neq 0$ then $z_0^{(1)}(\lambda)=0$ and this is a contradiction.

\noindent Case 2: If $c_2=d_2=0$ and $c_1 \neq 0$ we have 
\[y_n(\lambda)= c_1 y_n^{(1)}(\lambda)~ \forall n.\] Using the initial condition $y_0(\lambda)=0$ we have $c_1y_0^{(1)}(\lambda)=0$, this implies\\ $y_0^{(1)}(\lambda)=0$ which is a contradiction. Hence there are no solution of the difference equation (\ref{point_diff4}). By Case 1 and Case 2, we can deduce that no non-trivial solution exists for the system $Tx=\lambda x $. Hence $\lambda \not \in \sigma_p(T, \ell_p).$ Thus,
\[\sigma_p(T,\ell_p)=\left\lbrace \lambda \in \mathbb{C}: y_{0}^{(1)}(\lambda) =0\right\rbrace \cup \left\lbrace \lambda \in \mathbb{C}:{z_0}^{(1)}(\lambda)=0\right\rbrace .\] 
\end{proof}

\begin{remark}\label{remark1}
	The adjoint operator $ {T}^*:\ell_p^*\to \ell_p^* $, is represented by transpose of the matrix $T$ and dual of $\ell_p$ is isomorphic to $\ell_q$ where $\frac{1}{p}+ \frac{1}{q}=1$ and $1<q<\infty$. Similar as $T$, the operator $T^*$ can also be written as 
	\begin{equation*}
		T^*=T_0+K^t,
	\end{equation*}
where $K^t$ denotes the transpose of $K$ and $K^t$ is also a compact operator. Since $\sigma(T, \ell_{p})=\sigma(T^*, \ell_{p}^*)$, Theorem \ref{t4.11} implies $$\sigma(T_0, \ell_{p})\subseteq \sigma(T^*, \ell_{p}^*),$$ and using similar argument of the proof of Theorem \ref{t4.11} it can be obtain that $\sigma(T^*,\ell_{p}^*)\setminus \sigma(T_0,\ell_p)$ contains finite or countable number of eigenvalues of $T^*$ of finite type with no accumulation point in $\sigma(T^*,\ell_{p}^*)\setminus \sigma(T_0,\ell_p)$. Assuming similar hypothesis on the rate of convergence of sequences in Theorem \ref{poT},
 we can prove that $$\sigma_{ess}(T^*, \ell_{p}^*)\cap \sigma_p(T^*,\ell_{p}^*)=\emptyset$$
 and this implies, the point spectrum of $T^*$ is lying outside of the region \[[r_1-2s_1,r_1+2s_1]\cup[r_2-2s_2,r_2+2s_2].\]
Now similar as Theorem \ref{t4.12},
let $\{g_n^{(1)}(\lambda), g_n^{(2)}(\lambda)\}$ and $\{h_n^{(1)}(\lambda), h_n^{(2)}(\lambda)\}$ are the sets of fundamental solutions of the following difference equations respectively
\begin{eqnarray*}
	b_{2n-1}g_n+(a_{2n+1}-\lambda)g_{n+1}+c_{2n+1}g_{n+2}=0,\\
	b_{2n}h_n+(a_{2n+2}-\lambda)h_{n+1}+c_{2n+2}h_{n+2}=0,
\end{eqnarray*}
 which are obtained from $T^*f= \lambda f$,
$f \in\ell_{p}^*$ and $g_n(\lambda)=f_{2n-1}(\lambda)$, $h_n(\lambda)=f_{2n}(\lambda)$. Also, $g_0(\lambda)=h_0(\lambda)=0$.
This leads us to the following result 
\begin{equation*}
\sigma_p(T^*, \ell_{p}^*)=\left\lbrace \lambda \in \mathbb{C}:g_0^{(1)}(\lambda)=0\right\rbrace \cup \left\lbrace \lambda \in \mathbb{C}:h_0^{(1)}(\lambda)=0\right\rbrace.
\end{equation*}
Eventually, we obtain that
$$\sigma(T^*,{\ell_p^*})=[r_1-2s_1, r_1+2s_1]\cup [r_2-2s_2, r_2+2s_2] \cup S_2$$
where,
\begin{equation*}
	S_2=\left\lbrace \lambda \in \mathbb{C}:g_0^{(1)}(\lambda)=0\right\rbrace \cup \left\lbrace \lambda \in \mathbb{C}:h_0^{(1)}(\lambda)=0\right\rbrace. 
\end{equation*}
Since, $\sigma(T,\ell_{p})=\sigma(T^*,\ell_{p}^*)$ and $S_1,S_2$ both sets are disjoint from $[r_1-2s_1,r_1+2s_1]\cup[r_2-2s_2,r_2+2s_2]$ we have, $S_1=S_2$.
  
\end{remark}
 Using the observations in Remark \ref{remark1} and Proposition \ref{p3.1}, we can summarize all the results of spectrum and various spectral subdivisions of the operator $T$ in the following theorem.
	\begin{theorem}
		If the convergence of the sequences $\{a_{2n-1}\}, \ \{a_{2n}\}, \ \{b_{2n-1}\}, \ \{b_{2n}\},\\ \ \{c_{2n-1}\}$ and $\{c_{2n}\}$ are exponentially fast and
\[S_1 = \left\lbrace \lambda \in \mathbb{C}: y_{0}^{(1)}(\lambda) =0\right\rbrace \cup \left\lbrace \lambda \in \mathbb{C}:{z_0}^{(1)}(\lambda)=0\right\rbrace,\]		
		 then we have the following results,
	\begin{itemize}
		\item[(i)] The spectrum of $T$ on $\ell_p$ is
		\[\sigma(T,\ell_p)=[r_1-2s_1,r_1+2s_1]\cup[r_2-2s_2,r_2+2s_2]\cup S_1.\]
		\item[(ii)] The point spectrum of $T$ on $\ell_{p}$ is 
		 \[\sigma_p(T,\ell_p)=S_1.\]
		 \item[(iii)] The residual spectrum of $T$ on $\ell_{p}$ is 
		 	\[\sigma_r(T, \ell_p)=\emptyset.\]
		 \item[(iv)] The continuous spectrum of $T$ on $\ell_{p}$ is 
		 \[\sigma_c(T,\ell_{p})=[r_1-2s_1,r_1+2s_1]\cup[r_2-2s_2,r_2+2s_2].\]
		 \item[(v)] The essential spectrum of $T$ on $\ell_{p}$ is
		 \[\sigma_{ess}(T,\ell_{p})=[r_1-2s_1,r_1+2s_1]\cup[r_2-2s_2,r_2+2s_2].\]
		 \item[(vi)] The discrete spectrum of $T$ on $\ell_{p}$ is
		 \[\sigma_d(T, \ell_{p})=S_1.\]
		\item [(vii)] The compression spectrum of $T$ on $\ell_{p}$ is \[\sigma_{co}(T,\ell_p)=S_1.\]
		\item[(viii)] The approximate spectrum of $T$ on $\ell_{p}$ is \[\sigma_{app}(T,\ell_p)=[r_1-2s_1,r_1+2s_1]\cup[r_2-2s_2,r_2+2s_2]\cup S_1.\]
		\item[(ix)] The defect spectrum of $T$ on $\ell_{p}$ is \[\sigma_{\delta}(T,\ell_p)=[r_1-2s_1,r_1+2s_1]\cup[r_2-2s_2,r_2+2s_2]\cup S_1.\]
	\end{itemize}	
\end{theorem} 
\begin{proof} The proofs of the above statements are given below. 
	\begin{enumerate}
			\item[(i)] It is well known that $\sigma_p(T,\ell_p)\subseteq \sigma(T,\ell_p)$ and $\sigma(T_0,\ell_p)\subseteq \sigma(T,\ell_p)$. This implies, $$[r_1-2s_1,r_1+2s_1]\cup[r_2-2s_2,r_2+2s_2]\cup S_1\subseteq \sigma(T,\ell_p).$$ Also by using Theorem (\ref{t4.11}) we get, $$\sigma(T,\ell_p)\subseteq[r_1-2s_1,r_1+2s_1]\cup[r_2-2s_2,r_2+2s_2]\cup S_1.$$ Hence,
		$$\sigma(T,\ell_p)=[r_1-2s_1,r_1+2s_1]\cup[r_2-2s_2,r_2+2s_2]\cup S_1.$$
		\item[(ii)]  Result has been proved in Theorem \ref{t4.12}.
		\item [(iii)] We already aware of that $\sigma_r(T,\ell_p)=\sigma_p(T^*,\ell_p^*)\setminus\sigma_p(T,\ell_p)$.
		Hence, $$\sigma_r(T, \ell_p)=\emptyset.$$
		\item [(iv)] Spectrum of an operator is the disjoint union of point spectrum, residual spectrum and continuous spectrum. By using this result we can obtain the desired result.
		\item [(v)]
		The required result has been proved in Corollary \ref{ess1}.
		% Since compact perturbation does not effect the fredholmness of a operator so $\sigma_{ess}(T, \ell_{p})=\sigma_{ess}(T_0, \ell_{p})$. Therefore, by using Corollary \ref{ess} we have, 
		%\[\sigma_{ess}(T,\ell_{p})=[r_1-2s_1,r_1+2s_1]\cup[r_2-2s_2,r_2+2s_2].\]
		\item [(vi)] We already proved that the point spectrum of $T$ is disjoint from $[r_1-2s_1,r_1+2s_1]\cup[r_2-2s_2,r_2+2s_2]$, and by Theorem \ref{t4.11} we have, every element of $\sigma_{p}(T, \ell_{p})$ is of finite type. Hence,
		\[\sigma_d(T, \ell_{p})=\left\lbrace \lambda \in \mathbb{C}: y_{0}^{(1)}(\lambda) =0\right\rbrace \cup \left\lbrace \lambda \in \mathbb{C}:{z_0}^{(1)}(\lambda)=0\right\rbrace.\]
		\item [(vii)] From part (e) of Proposition \ref{p3.1}, the desired result is obvious.
		\item [(viii)] Clearly, $$\sigma_{app}(T, \ell_{p})\subseteq \sigma(T,\ell_p)=[r_1-2s_1,r_1+2s_1]\cup[r_2-2s_2,r_2+2s_2]\cup S_1.$$ Also, we know that point spectrum is always a subset of approximate point spectrum.  By using this fact and with the help of part (g) of Proposition \ref{p3.1}, we have 
		\[[r_1-2s_1,r_1+2s_1]\cup[r_2-2s_2,r_2+2s_2]\cup S_1 \subseteq \sigma_{app}(T, \ell_{p}).\]
		Hence,\[\sigma_{app}(T, \ell_{p})=\sigma(T, \ell_{p})=[r_1-2s_1,r_1+2s_1]\cup[r_2-2s_2,r_2+2s_2]\cup S_1.\]
		\item [(ix)] From part (c) of Proposition \ref{p3.1}, we have 
		\[\sigma_{app}(T^*, \ell_{p}^*)=\sigma_{\delta}(T,\ell_{p}).\]
		Clearly, $$\sigma_{app}(T^*, \ell_{p}^*)\subseteq \sigma(T^*,\ell_{p}^*)=[r_1-2s_1,r_1+2s_1]\cup[r_2-2s_2,r_2+2s_2]\cup S_1.$$ 
		Then $S_1\subseteq\sigma_{app}(T^*, \ell_{p}^*),$ follows from the fact that $\sigma_{p}(T^*, \ell_{p}^*)\subseteq\sigma_{app}(T^*, \ell_{p}^*)$. By using this fact and with the help of part (g) of Proposition \ref{p3.1}, we have 
		\[[r_1-2s_1,r_1+2s_1]\cup[r_2-2s_2,r_2+2s_2]\cup S_1 \subseteq \sigma_{app}(T^*, \ell_{p}^*).\] Therefore, $\sigma_{app}(T^*, \ell_{p}^*)=\sigma(T,\ell_{p})$. Hence,
		\[\sigma_{\delta}(T,\ell_p)=[r_1-2s_1,r_1+2s_1]\cup[r_2-2s_2,r_2+2s_2]\cup S_1.\]
	\end{enumerate}
\end{proof}

\begin{remark}
One interesting observation of the above theorem is the relation $\sigma_p(T,\ell_p) = \sigma_d(T,\ell_p)=\sigma_{co}(T,\ell_p) = S_1$ holds. In other words, all the eigenvalues are of finite type and range of $T - \lambda I$ is not dense in $\ell_p$ for any eigenvalue $\lambda.$
\end{remark}
\section*{Declarations}
Conflict of Interest: All the authors declare that they have no conflict of interest.

\end{document}